\documentclass[11pt]{article}

\usepackage[margin=1.08in]{geometry}
\usepackage{amsmath,amssymb,amsthm,mathtools}
\usepackage{array,booktabs,tabularx}
\usepackage{microtype}
\usepackage{enumitem}
\usepackage{natbib}
\usepackage{xcolor}
\usepackage{comment}
\usepackage[colorlinks=true,allcolors=blue]{hyperref}
\usepackage{setspace}

\newtheorem{theorem}{Theorem}

\newtheorem{lemma}[theorem]{Lemma}

\theoremstyle{remark}

\newcommand{\bbP}{\mathbb P}
\newcommand{\cind}{\stackrel{d}{\longrightarrow}}
\newcommand{\Dirichlet}{\operatorname{Dirichlet}}

\newcommand{\Reals}{\mathbb R}

\newcommand{\DP}{\operatorname{DP}}
\newcommand{\Gam}{\operatorname{Gam}}

\newcommand{\Normal}{\operatorname{Normal}}
\newcommand{\E}{\mathbb{E}}
\newcommand{\iid}{\mathrel{\overset{\mathrm{iid}}{\sim}}}
\newcommand{\indep}{\stackrel{\text{indep}}{\sim}}

\newcommand{\sP}{\mathcal P}
\newcommand{\Var}{\operatorname{Var}}

\title{A Simple Example of Bayesian Nonparametric Inconsistency}
\author{Antonio R. Linero}
\date{}

\begin{document}
\maketitle

\begin{abstract}
  I present a very simple example in which a full-support prior over distribution functions has an inconsistent posterior, which I believe has instructive value. The example is $Y_i \iid P_0$ under the Bayesian hierarchical model $[Y_i \mid P] \iid P$ and $P \sim \int \DP(\alpha, H) \, \pi_\alpha(\alpha) \ d\alpha$ when $\pi_\alpha(\cdot)$ has exponential (or heavier) tails, where $\DP(\alpha, H)$ denotes a Dirichlet process with concentration parameter $\alpha$ and mean $H(\cdot)$; on the other hand, consistency is obtained under a light-tailed prior. This unifies and generalizes a consistency result of \citet{freedman1983inconsistent} with an inconsistency result described by \citet{ferguson1992bayesian}.
\end{abstract}

\doublespacing

\section{Introduction}

That one cannot be completely reckless with prior specification in Bayesian nonparametric problems is well known. The purpose of this note is to provide an accessible demonstration of posterior inconsistency despite the prior having full support (in the weak topology). I believe this example may be useful as a pedagogical tool because (i) it is very elementary and easy to see why things go wrong, but (ii) it is a mistake that one can imagine someone actually making. Indeed, I have actually seen this example occur in the wild as a reviewer! Because of this, I also hope this work may also prevent similar mistakes from occurring in the future.

The example is a mixture of Dirichlet process priors \citep{Antoniak1974} of the form
\begin{align}
  \label{eq:model}
  \alpha \sim \pi_\alpha,
  \quad [P \mid \alpha] \sim \DP(\alpha, H),
  \quad [Y_1, Y_2, \ldots \mid P] \iid P.
\end{align}
Let $P_0$ be any continuous distribution such that, in reality, $Y_i \iid P_0$, and assume that $H \ne P_0$. I will show that the posterior of $P$ is inconsistent as $N \to \infty$ when $\pi_\alpha(\cdot)$ has exponential (or heavier) tails. For example, if $\alpha \sim \Gam(a_\alpha, b_\alpha)$ then the posterior converges to a point mass at $\alpha^\star P_0 + (1 - \alpha^\star) H$ for some $0 < \alpha^\star < 1$; if, on the other hand, $\alpha$ has a half-Cauchy prior then the posterior converges to a point mass at $H(\cdot)$. Posterior consistency is also possible under a light-tailed prior, such as when $\alpha$ has a half-normal prior. We prove the following result:

\begin{theorem}
  \label{thm:main}
  Let $Q_t = \frac{t}{1 + t} H + \frac{1}{1 + t} \, P_0$ and let $\rho(\cdot, \cdot)$ metrize convergence in distribution. Define $\Pr(\cdot \mid Y_{1:N})$ to be the posterior of $P$ under \eqref{eq:model}. Under the setting described in this paper, for every $\epsilon > 0$ we have $\Pr\{\rho(P, Q_t) > \epsilon \mid Y_{1:N}\} \to 0$ with $P_0$-probability $1$, where $t = 0$ when $\pi_\alpha(\alpha)$ has light tails, $t$ is given in Theorem~\ref{thm:alpha} when $\pi_\alpha(\alpha)$ has exponential tails, and $t = \infty$ when $\pi_\alpha(\alpha)$ has heavy tails.
\end{theorem}

See Section~\ref{sec:formal} for the formal definitions of light, exponential, and heavy tails. Examples of posterior consistency of this form are important because they inoculate users from the naive expectation that a prior having ``large support'' will imply posterior consistency. The posterior here is inconsistent despite the fact that the prior has ``full support'' when $H(\cdot)$ has full support on $\Reals^d$, in the sense that $\Pr\{\rho(P, P_0) < \epsilon\} > 0$ for all $\epsilon > 0$ and all $P_0$, i.e., draws of $P$ from the prior can be ``arbitrarily close'' to $P_0$.

\subsection{Comments}

\citet{DiaconisFreedman1986a,DiaconisFreedman1986b} provide the most well-known examples of posterior inconsistency. The result in this paper should be intuitive to experts, and it is, at the very least, heavily hinted at by prior work. I am mostly surprised that it is not more \emph{well-known}, as most people I have mentioned this example to were not aware that such a simple example exists. Rather, most people seem to know about the general findings of \citet{DiaconisFreedman1986a,DiaconisFreedman1986b}, and also understand that proving posterior consistency is non-trivial with respect to stronger metrics, but do not know either the details or the proofs supporting the examples; based on this work, it is natural to assume (incorrectly) that posterior inconsistency only occurs in highly contrived or pathological cases.

\citet{ferguson1992bayesian} essentially describe a special case of Theorem~\ref{thm:main} (use the prior $P \sim 0.5 \DP(\alpha, H) + 0.5 \delta_H$) and explain briefly why the posterior is inconsistent --- the argument is very short, but the prior may appear contrived. \citet{freedman1983inconsistent} show that consistency is obtained under a bounded prior with $\Pr_{\pi_{\alpha}}(\alpha > M) = 0$ for some finite $M > 0$; their argument is for discrete data, but state that it is straight forward to extend the result to continuous data. \citet{ghosal1999consistency} also provide substantial discussion of the main points raised in this example. I am unaware of any results for unbounded-but-finite priors analogous to Theorem~\ref{thm:main}, and in particular am unaware of prior work that obtains convergence to a mixture when $\alpha$ has exponential tails.


On the positive side, sufficient conditions that guarantee posterior consistency are well-known. If we instead demand that $P_0$ lie in the \emph{Kullback-Leibler support} of the prior, i.e., $\Pr\{\operatorname{KL}(P_0 \| P) < \epsilon\} > 0$ for all $\epsilon > 0$ where $\operatorname{KL}(P_0 \| P) = \int \log (\frac{dP_0}{dP}) dP_0$, then a classic result {due to \citet{schwartz1965bayes} states that $\Pr\{\rho(P, P_0) > \epsilon \mid Y_{1:N}\} \to 0$. This result is not applicable here because $dP_0 / dP$ almost surely does not exist ($P$ is discrete but $P_0$ is continuous).

\subsection{Where This Might Occur in Practice}

The setting in which I saw this issue arise was one in which the analyst wanted to ``express ignorance'' about the weight assigned to a prior guess. The setting was more complicated than \eqref{eq:model}. Rather than having a single $P$, the problem concerned many distributions $P_1, P_2, \ldots, P_K$ all sharing a $\DP(\alpha, H_\theta)$ prior, where $\{H_\theta : \theta \in \Theta\}$ is a parametric family of models.

Distilling this idea, consider instead the hierarchical model
\begin{align*}
  (\alpha, \theta) \sim \pi_{\alpha,\theta}(\alpha, \theta), \qquad
  [P_1, \ldots, P_K \mid \alpha, \theta] \iid \DP(\alpha, H_\theta), \qquad
  [Y_{ik} \mid \alpha, \theta, \{P_k\}_{k = 1}^K] \indep P_k.
\end{align*}
When the $P_k$'s are the target of inference, the conditional posterior is given by
\begin{align*}
  P_k \sim \DP\left( \alpha + N_k , \frac{\alpha}{\alpha + N_k} H_\theta + \frac{N_k}{\alpha + N_k} \, \bbP_k\right),
  \quad \text{where} \quad
  \bbP_k = \frac{1}{N_k} \sum_{i = 1}^{N_k} \delta_{Y_{ik}}.
\end{align*}
Learning $\alpha$ and $\theta$ together would then be very useful, because it would allow us to learn both the target $H_\theta$ toward which to shrink the distributions and the value of $\alpha$ that controls the amount of shrinkage applied. This is a very natural generalization of what is done in hierarchical models such as $Y_{ij} \indep \Normal(\mu_j, \sigma^2_j)$ and $\mu_j \sim \Normal(m, \sigma^2_\mu)$; learning $m$ lets us learn where to shrink the $\mu_j$'s, while learning $\sigma^2_\mu$ lets us learn how much shrinkage to apply \citep{gelman2007data}.

This type of data-adaptive shrinkage has been shown to be extremely useful!
Assuming that group-level densities are not of interest, it is also much more pragmatic to use a raw Dirichlet process rather than something based on a Dirichlet process \emph{mixture} \citep{EscobarWest1995}, which would not have this inconsistency problem. Dirichlet process mixtures and other advanced tools like the hierarchical Dirichlet process \citep[HDP,][]{teh2006hierarchical} and the nested Dirichlet process \citep[NDP,][]{rodriguez2008nested} introduce many additional hyperparameters and are not trivial to fit reliably.

The lesson is not that it is misguided to learn $\alpha$ empirically. It is just that you cannot do it with Bayesian reasoning (which is fair, considering that the prior is incongruous with continuous data). One possibility is to choose $\alpha$ to reduce some measure of out-of-sample risk. For example, rather than specifying a prior, we could choose $\alpha$ to minimize a held-out continuous ranked probability score (CRPS) \citep{GneitingRaftery2007} as
\begin{align*}
  (\widehat \alpha, \widehat \theta)
  =
  \underset{\alpha>0, \theta}{\arg\min}
  \sum_{k=1}^{K}\sum_{i=1}^{N_k}
  \int_{\mathbb R}
  \left[
  \frac{\alpha H_\theta(y)+\sum_{j\ne i} 1(Y_{jk}\le y)}
  {\alpha+N_k-1}
  - 1(Y_{ik}\le y)
  \right]^2 \, dy.
\end{align*}
This would instead tune $\alpha$ for recovering the best predictive distribution for the observations $Y_{ik}$ as measured by CRPS.

\section{Some Background}

In this section, we define relevant terms and introduce notation. Let $P_0$ and $H$ be continuous probability distributions on $\Reals^d$ with densities $p_0(\cdot)$ and $h(\cdot)$, such that $h(Y_i) > 0$ with $P_0$-probability 1. Let $\sP(\Reals^d)$ denote the set of probability distributions on $\Reals^d$.

We say that a random distribution $P$ is a \emph{Dirichlet process} with mean $H(\cdot)$ and concentration parameter $\alpha$ if, for any (measurable) partition $\{A_1, \ldots, A_M\}$ of $\Reals^d$, we have $(P(A_1), \ldots, P(A_M)) \sim \Dirichlet(\alpha \, H(A_1), \ldots, \alpha \, H(A_M))$. To denote this, we write $P \sim \DP(\alpha, H)$. We refer to $H(\cdot)$ as the mean and $\alpha$ as the concentration parameter because they relate to mean and variance via
\begin{align*}
  \E\{P(A_i)\} = H(A_i) \qquad \text{and} \qquad
  \Var\{P(A_i)\} = \frac{H(A_i) \{1 - H(A_i)\}}{\alpha + 1},
\end{align*}
where $\E(\cdot)$ and $\Var(\cdot)$ are the expectation and variance operators, respectively.

The Dirichlet process has the desirable property of being a conjugate prior for iid sampling: if $[Y_1, \ldots, Y_N \mid P] \sim P$ and $P \sim \DP(\alpha, H)$, then the posterior is the Dirichlet process
\begin{align*}
  [P \mid Y_1, \ldots, Y_N] \sim \DP\left( \alpha + N, \frac{\alpha}{\alpha + N} \, H + \frac{N}{\alpha + N} \, \bbP_N \right)
  \quad \text{where} \quad
  \bbP_N = \frac{1}{N} \sum_{i = 1}^N \delta_{Y_i}(\cdot).
\end{align*}
Here, $\alpha$ functions as a ``prior number of observations'' that belief in the mean is worth, and that the update consists of a weighted combination of our prior expectation ($P \approx H$) and the maximum likelihood estimator $(P \approx \bbP_N)$, which is analogous to many other Bayes estimates under conjugate priors.

Another desirable property of the Dirichlet process is that it has large support in the following sense. Let $\rho(P, Q)$ metrize convergence in distribution on $\sP(\Reals^d)$, i.e., $\rho(P_N, Q) \to 0$ implies $P_N \cind Q$, and suppose $H$ has full support on $\Reals^d$. Then $\Pr\{d(P, Q) < \epsilon\} > 0$ for all $\epsilon$ and $Q \in \sP(\Reals^d)$. That is, samples from the Dirichlet process can be arbitrarily close to any distribution on $\Reals^d$ \citep{Majumdar1992}. 

\citet{Ferguson1973} introduced the Dirichlet process and showed that it is almost surely discrete.
While every (say) continuous distribution can be in the weak support of the Dirichlet process, draws from the Dirichlet process itself are never continuous. This turns out to be quite important, as this implies that draws $[Y_1, Y_2, \ldots, Y_N \mid P] \iid P$ are likely to include \emph{ties}. When $H(\cdot)$ is continuous and $\alpha \sim \pi_\alpha$, the posterior of density $\alpha$ is given by
\begin{align}
  \label{eq:urn}
  \pi(\alpha \mid Y_{1:N}) \propto \frac{\pi_\alpha(\alpha) \alpha^k \, \Gamma(\alpha)}{\Gamma(\alpha + N)}
\end{align}
where $k$ is the number of distinct values among $Y_{1:N}$ \citep[see, e.g.,][]{EscobarWest1995}.

\section{The High Level Argument}
\label{sec:high}

The mixture of Dirichlet processes (MDP) \citep{Antoniak1974} leads to the posterior distribution for $P$ of the form
\begin{align}
  \label{eq:mdp}
  \int \DP\left(dP \mid \alpha + N, \frac{\alpha}{\alpha + N} H + \frac{N}{\alpha + N} \bbP_N \right) \, \pi(\alpha \mid Y_{1:N}),
\end{align}
where $\pi(\alpha \mid Y_{1:N})$ is the marginal posterior distribution of $\alpha$. 
For continuous data, by \eqref{eq:urn}, we have almost surely that $K_N = N$, so the posterior density of $\alpha$ is $\pi(\alpha \mid Y_{1:N}) \propto \pi_\alpha(\alpha) \frac{\alpha^N \Gamma(\alpha)}{\Gamma(\alpha + N)}$. 
The likelihood contribution for $\alpha$ is
\begin{align}
  \label{eq:sumlog}
  \log L_N(\alpha)
  = \sum_{j = 0}^{N - 1}  \log \frac{\alpha}{\alpha + j}
\end{align}
This expression is increasing in $\alpha$ as, under the model, observing no ties is evidence in favor of a large concentration parameter.

A simple way to see how the choice of prior determines the posterior is to consider the maximum a posteriori estimate of $T_N=\alpha/N$. Let $V(\alpha)=-\log\pi_\alpha(\alpha)$ and, up to a constant that does not depend on $t$, write the log-posterior density of $T_N$ and the normalized score as
\begin{align*}
  \ell_N(t)=\log L_N(Nt)-V(Nt) \quad \text{and} \quad s_N(t)
  = \left . \frac{\partial \log L_N(\alpha)}{\partial \alpha}\right|_{\alpha = Nt}
  =\frac{1}{N}\sum_{j=1}^{N-1}\frac{j}{t(Nt+j)},
\end{align*}
and note that $s_N(t) \to s(t) = \frac{1}{t}-\log\left(1+\frac{1}{t}\right)$ as $N \to \infty$ for every fixed $t > 0$ by a Riemann-sum argument.  We now examine the score equation $s_N(t)=V'(Nt)$ directly for three familiar priors.

\paragraph{Exponential prior.}

If $\pi_\alpha(\alpha)=c\exp(-c\alpha)$, then $V'(\alpha)=c$. The limiting score equation is therefore
\begin{align}
  \label{eq:tc}
  \frac{1}{t} - \log\left(1 + \frac{1}{t}\right) = c.
\end{align}
Because the left-hand side decreases continuously from infinity to zero, there is a unique solution $t_c \in(0,\infty)$, and so under mild conditions we should expect that the MAP satisfies $\widehat T_N\to t_c$.

\paragraph{Half-normal prior.}

If $\pi_\alpha(\alpha)\propto\exp\{-\alpha^2/(2\sigma^2)\}$ for $\alpha>0$, then $V'(Nt)=Nt/\sigma^2$. At every fixed $t>0$ this term diverges, so the solution of the score equation must tend to zero. Putting $t=c/\sqrt{N}$ gives
\begin{math}
  \frac{s_N(c/\sqrt{N})}{\sqrt{N}}\longrightarrow\frac{1}{c},
  \text{ and } 
  \frac{V'(c\sqrt{N})}{\sqrt{N}}\longrightarrow\frac{c}{\sigma^2}.
\end{math}
Equating the limits gives $c=\sigma$, so $\widehat T_N\sim\sigma/\sqrt{N}\to0$.

\paragraph{Cauchy prior.}

Take a Cauchy density with scale $\tau$, restricted to $\alpha>0$, so that $\pi_\alpha(\alpha)\propto\{1+(\alpha/\tau)^2\}^{-1}$. Then
\begin{math}
  V'(Nt)=\frac{2Nt}{\tau^2+N^2t^2}\longrightarrow0
\end{math}
at every fixed $t>0$. Since $s(t)>0$, the solution of the score equation must tend to infinity. In fact, putting $t=cN$ gives
\begin{math}
  N^2s_N(cN)\longrightarrow\frac{1}{2c^2},
  \text{ and }
  N^2V'(cN^2)\longrightarrow\frac{2}{c}.
\end{math}
Equating the limits gives $c=1/4$, so $\widehat T_N\sim N/4\to\infty$. 

\bigskip

Thus the three MAP estimates tend to a positive constant, zero, and infinity,
respectively.
Assuming this concentration, the three outcomes for $P$ follow from the center of the conditional Dirichlet-process posterior,
\begin{align*}
  [P \mid \alpha, Y_{1:N}] \sim \DP\left(N + N T_N, \frac{T_N}{1+T_N}H + \frac{1}{1+T_N}\bbP_N\right).
\end{align*}
Because $\bbP_N$ converges weakly to $P_0$, the limiting center is $(P_0+t_c H)/(1+t_c)$ under the exponential prior, $P_0$ under the half-normal prior, and $H$ under the Cauchy prior. The concentration of the conditional Dirichlet process is $N(1+T_N)$, so its fluctuations around this center also vanish. Thus the exponential prior produces an inconsistent nontrivial mixture, the half-normal prior produces weak posterior consistency, and the Cauchy prior causes the posterior to concentrate at $H$.

\paragraph{What goes wrong}

The fundamental problem with attempting to learn $\alpha$ from the data is that the fact that there are no ties in the data provides evidence that $\alpha$ is large, which induces more shrinkage towards $H$. Indeed, the value of $\alpha$ that is most consistent with observing no ties is $\alpha \to \infty$, as this would make $\Pr(K_N = N \mid \alpha) \to 1$. The role played by the prior is therefore only to express a preference for some finite values of $\alpha$. The likelihood prevents taking $\alpha$ too small, and if the tails of $\pi_\alpha$ are too heavy then the prior expresses a preference for $\alpha \asymp N$ or $\alpha \gg N$. For thin tails, the prior prefers $\alpha \ll N$ so that, while the posterior concentrates on increasingly large values of $\alpha$, these values are small enough that $\bbP_N$ receives all of the mass asymptotically.

\paragraph{Other Inconsistencies}

Posterior inconsistency with a prior on $\alpha$ does \emph{not} occur if we replace the mixture of Dirichlet processes with a \emph{Dirichlet process mixture} of the form
\begin{align*}
  \alpha \sim \pi_\alpha,
  \qquad
  [P \mid \alpha] \sim \DP(\alpha, H),
  \qquad 
  [Y_1, Y_2, \ldots \mid P] \iid f(y \mid P) = \int \kappa(y; \theta) \ P(d\theta),
\end{align*}
for some density in a parametric family $\kappa(\cdot; \theta)$. Conditions for posterior consistency with respect to both weak convergence and total variation for Dirichlet process mixtures are given, for example, by \citet{GhosalGhoshRamamoorthi1999}. Placing a prior $\alpha \sim \pi_\alpha(\alpha)$ is also routine in these settings \citep{EscobarWest1995}. Interestingly, and completely separately from the source of inconsistency in this work, it is still possible to go wrong here by placing an \emph{improper} prior on $\alpha$. Taking $\pi_\alpha(\alpha) = 1(\alpha > 0)$, for example, leads to an improper posterior \citep{vicentini2025prior}. Fortunately, we are not aware of anyone falling into this particular trap.

\section{Formal Arguments}
\label{sec:formal}

The argument in Section~\ref{sec:high} is heuristic, so for the sake of completeness I provide details below. To show that the posterior assigns negligible mass to a set $A_N$, we control
\begin{align*}
  \pi_\alpha(\alpha \mid Y_{1:N})
  =
  \frac{L_N(\alpha) \, \pi_\alpha(\alpha)}{\int L_N(a) \, \pi_\alpha(a) \ da}
  =
  \frac{L_N(\alpha) \, \pi_\alpha(\alpha)}{Z_N}
\end{align*}
by (i) lower-bounding the denominator and (ii) upper-bounding the numerator on the set $A_N$.

First, we need an asymptotic approximation to the likelihood contribution. Define
\begin{align}
  \label{eq:I}
  I(t) = \int_0^1 \log\left( 1 + \frac x t \right) \ dx
  = (1 + t) \log \left( 1 + \frac 1 t \right) - 1.
\end{align}
The function $I(\cdot)$ is strictly decreasing and convex, $I(t) \to \infty$ as $t \downarrow 0$, and $I(t) \to 0$ as $t \to \infty$.

\begin{lemma}[Likelihood asymptotics]\label{lem:likelihood}
  For every $N \ge 2$, $L_N$ is increasing on $(0,\infty)$. Moreover, for every $0<r<R<\infty$,
\begin{align*}
  \sup_{t\in[r,R]}
  \left|\frac 1 N\log L_N(Nt)+I(t)\right| \longrightarrow 0,
\end{align*}
and, for all $a > 0$,
\begin{math}
  0\le -\log L_N(a)\le \frac{N(N-1)}{2a}.
\end{math}
\end{lemma}

\begin{proof}
  Monotonicity can be verified directly. With $a = Nt$, the right-hand side of \eqref{eq:sumlog}, divided by $N$, is a Riemann sum for \eqref{eq:I}; joint continuity of $(x,t) \mapsto \log(1+x/t)$ on $[0,1]\times[r,R]$ gives uniform convergence. Finally, $-\log L_N(a) \le N (N - 1) / (2a)$ follows from $\log(1 + x) \le x$.
\end{proof}

Let $\mu_N$ denote the posterior law of $T_N = \alpha / N$ given $[K_N = N]$. We consider three different tail regimes for $\pi_\alpha$. Recall that $V(\alpha) \stackrel{\text{def}}{=} -\log \pi_\alpha(\alpha)$.

\begin{description}
\item[Light Tail (L)] Either $\pi_\alpha$ has bounded support, or $\pi_\alpha(\cdot)$ is log-concave for sufficiently large $\alpha$ with $V(\alpha) / \alpha \to \infty$ as $\alpha \to \infty$.
\item[Exponential Tail (E)] There exists a $c \in (0, \infty)$ such that $V(\alpha) / \alpha \to c$ as $\alpha \to \infty$.
\item[Heavy Tail (H)] $V(\alpha) / \alpha \to 0$ as $\alpha \to \infty$.
\end{description}

\begin{theorem}
  \label{thm:alpha}
  Let $\mu_N$ be the distribution of $T_N = \alpha / N$ under the posterior $\pi_\alpha(\alpha \mid Y_{1:N}) \propto \pi_\alpha(\alpha) \frac{\alpha^N \, \Gamma(\alpha)}{\Gamma(\alpha + N)}$. Then, $T_N \to 0$ in $\mu_N$-probability under (L), $T_N \to t_c$ in $\mu_N$-probability under (E), and $T_N \to \infty$ in $\mu_N$-probability under (H).
\end{theorem}

\begin{proof}
  Let $Z_N$ denote the normalizing constant of the posterior.
  
  \bigskip
  \noindent \emph{Light tails.} When $\pi_\alpha(\cdot)$ has bounded support, the result is immediate. Otherwise, fix $\epsilon > 0$ and $\delta = \epsilon / 2$. For all large $N$, $V(\alpha)$ is convex and (by convexity and the fact that $V(\alpha) / \alpha \to \infty$) increasing on $[\delta N, \infty)$. Define the secant slope 
  \begin{math}
    s_N = \frac{V(\epsilon N) - V(\delta N + 1)}{\epsilon N - \delta N - 1}.
  \end{math}
  Convexity and $V(\alpha) / \alpha \to \infty$ imply that $s_N \to \infty$.

  For $\alpha \ge \epsilon N$, convexity of $V(\alpha)$ gives
  \begin{math}
    V(\alpha) \ge V(\epsilon N) + s_N(\alpha - \epsilon N).
  \end{math}
  Hence, 
  \begin{align*}
    \int_{\epsilon N}^{\infty} L_N(\alpha) \, \pi_\alpha(d\alpha)
    \le \int_{\epsilon N}^\infty e^{-V(\alpha)} \ d\alpha
    \le \frac{e^{-V(\epsilon N)}}{s_N}.
  \end{align*}
  On the other hand, (eventual) monotonicity of $L_N(\alpha)$ and $V(\alpha)$ gives for sufficiently large $N$ that
  \begin{align*}
    Z_N \ge
    \int_{\delta N}^{\delta N + 1} L_N(\alpha) e^{-V(\alpha)} \ d\alpha
    \ge L_N(\delta N) e^{-V(\delta N + 1)}.
  \end{align*}
  By Lemma~\ref{lem:likelihood}, $-\log L_N(\delta N) = N I(\delta) + o(N)$. Combining the above two displays gives
  \begin{align*}
    \log \mu_N([\epsilon, \infty))
      \le -s_N (\epsilon N - \delta N - 1) + N I(\delta) + o(N) - \log s_N.
  \end{align*}
  Because $s_N \to \infty$, the first term dominates and so $\log \mu_N([\epsilon ,\infty)) \to -\infty$. 

  \bigskip
  \noindent \emph{Exponential Tails.} For $z > 0$, define $F_c(z) =I(z) + c \, z$. Because $I'(z) = \log(1 + z^{-1}) - z^{-1}$ and $I''(z) = z^{-2}(1 + z)^{-1} > 0$, $F_c(z)$ is strictly convex. We also note that $F'_c(z) \to -\infty$ as $z \to 0$ and $F'_c(z) \to c$ as $z \to \infty$; $F_c(z)$ therefore has a unique minimizer $t_c$ characterized by \eqref{eq:tc}. Fix a neighborhood $U$ of $t_c$. We aim to show
  \begin{math}
    \frac{\int_{\alpha / N \notin U} L_N(\alpha) \, \pi_\alpha(\alpha) \ d\alpha}{\int L_N(\alpha) \, \pi_\alpha(\alpha)} \to 0
  \end{math}
  as $N \to \infty$. For every compact interval $[r, R] \subseteq (0, \infty)$, Lemma~\ref{lem:likelihood} and $V(\alpha) / \alpha \to c$ imply that
  \begin{align}
    \label{eq:likelihood-bound}
    \sup_{t \in [r, R]} \left| \frac{1}{N} \log \{L_N(N t) \, \pi_\alpha(N t)\} + F_c(t) \right| \to 0.
  \end{align}
  Let $m = F_c(t_c)$ and choose $0 < r < t_c < R$ and $\eta > 0$ such that
  \[
    I(r) > m + 4\eta,
    \qquad \frac{cR}{2} > m + 4\eta,
    \qquad
    \inf_{z \in [r,R] \cap U^c} F_c(z) > m + 4\eta.
  \]
  Choose a closed interval $J = [\ell, u] \subseteq U \cap (r, R)$ containing $t_c$ in its interior and small enough that $\sup_{z \in J} F_c(z) < m + \eta$. By \eqref{eq:likelihood-bound}, for all large $N$, we have
  \begin{align*}
    Z_N \ge \int_{\alpha / N \in J} L_N(\alpha) \, \pi_\alpha(\alpha) \ d\alpha
    \ge N (u - \ell) e^{-N (m + 2\eta)}.
  \end{align*}
  To control the numerator $\int_{\alpha / N \in U^c} L_N(\alpha) \, \pi_\alpha(\alpha) \ d\alpha$, we bound it by
  \begin{align*}
    \int_{\alpha / N \le r} L_N(\alpha) \, \pi_\alpha(\alpha) \ d\alpha + 
    \int_{\alpha / N \ge R} L_N(\alpha) \, \pi_\alpha(\alpha) \ d\alpha + 
    \int_{\alpha / N  \in [r, R] \cap U^c} L_N(\alpha) \, \pi_\alpha(\alpha) \ d\alpha.
  \end{align*}
  The first term is bounded by $L_N(Nr) = e^{- N I(r) + o(N)} \le e^{-N(m + 3\eta)}$ for
  large $N$. For large $N$ as well, because $V(\alpha) / \alpha \to c$, we have
  $\pi(\alpha) \le e^{-c \alpha  / 2}$ for large $\alpha$, so the contribution
  of the second term is at most $C \, e^{-c N R / 2} \le C e^{-N(m + 4\eta)}$ for some constant $C$.
  Finally, by \eqref{eq:likelihood-bound}, the third term is bounded by $NR e^{-N\inf_{z \in [R, U] \cap U^c} F_c(z) + o(N)} \le N R
  e^{-N(m + 3 \eta)}$. Each of these terms converges to $0$ faster than
  $Z_N$, giving $\mu_N(U^c) \to 0$.

  \smallskip
  \noindent
  \emph{Heavy tails.} We will prove that, for some sequences $a_N$ and $b_N = o(N)$, we have
  \begin{align*}
    Z_N \ge \int_{a_N}^{a_N + 1} L_N(\alpha) \, \pi_\alpha(\alpha) \ d\alpha \ge
    e^{-b_N} = e^{-o(N)}.
  \end{align*}
  The result follows from this combined with the fact that, by monotonicity of $L_N(\cdot)$ and Lemma~\ref{lem:likelihood}, we have
  \begin{align*}
    \int_0^{M \, N} L_N(\alpha) \, \pi_\alpha(\alpha) \ d\alpha
    \le L_N(M N)
    = \exp\{-N I(M) + o(N)\}.
  \end{align*}
  The following argument uses the tail properties of $V(\alpha)$ to construct an appropriate $(a_N, b_N)$.

  Let $r(x) = \sup_{u \ge x} |V(u)| / u$ and note that $r(x) \to 0$ as $x \to \infty$. Choose integers $M(k) \uparrow \infty$ so that $r(x) \le k^{-2}$ for $x \ge M(k)$ and define the quantity $s_N = \max A_N$ where $A_N = \{k : 1 \le k \le \lfloor \sqrt N \rfloor : Nk \ge M(k)\}$ for $N$ large enough that this set is non-empty (define $s_N$ arbitrarily for the finitely many remaining values of $N$). For any fixed $k$, note that eventually we have $k \le \sqrt N$ and $Nk \ge M(k)$; consequently $s_N \ge k$ for sufficiently large $N$, i.e., $s_N \to \infty$. Moreover, because $s_N \in A_N$, we also have $r(N s_N) \le s_N^{-2}$ and $s_N r(N \, s_N) \le s_N^{-1} \to 0$. Take $a_N = N \, s_N$.

  Now, for $a \in [a_N, a_{N} + 1]$ we have
  \begin{math}
    V(a) \le a \frac{|V(a)|}{a} \le (a_{N} + 1) r(a_N) = o(N).
  \end{math}
  Additionally, by Lemma~\ref{lem:likelihood}, we have 
  \begin{math}
    - \log L_N(a) \le \frac{N (N - 1)}{2 a_N} \le \frac{N}{2 s_N} = o(N).
  \end{math}
  These bounds give $L_N(\alpha) \, e^{-V(\alpha)} \ge e^{-b_N} = e^{o(N)}$ where $b_N = (a_N + 1) r(a_N) + N / (2 \, s_N)$.
\end{proof}



We now prove Theorem~\ref{thm:main}. We will use the metric
\begin{align*}
  \rho(P, Q) = \sum_{k = 1}^\infty 2^{-k} |P f_k - Q f_k|,
\end{align*}
where $Pf = \int f \ dP$ and $\{f_k : k = 1,2,\ldots\}$ is a countable family of bounded continuous functions that determine weak convergence with $\|f_k\|_\infty \le 1$. Such a family exists for $\Reals^d$ (the Gaussian bump functions $\{e^{-\|y - \nu\|_2^2/(2\sigma^2)} : \sigma \in \mathbb Q_+, \nu \in \mathbb Q^d\}$ suffice), and the resulting metric metrizes weak convergence with $0 \le \rho \le 2$, as $\rho(P_n, Q) \to 0$ holds if-and-only-if $P_n f_k \to Q f_k$ for all $k$.

\begin{proof}
  Define $Q_{N,\alpha} = \frac{N \bbP_N + \alpha H}{N + \alpha}$. Conditional on $\alpha$, Tonelli's theorem Cauchy--Schwarz give
  \begin{align*}
    \E\{\rho(P, Q_{N, \alpha}) \mid \alpha, Y_{1:N}\}
    =
    \sum_k 2^{-k} \E\{|P f_k - Q_{N, \alpha} f_k| \mid \alpha, Y_{1:N}\}
    \le \sum_k 2^{-k} \sqrt{\Var(P f_k \mid \alpha, Y_{1:N})}.
  \end{align*}
  Properties of the Dirichlet process give $\Var(P f_k \mid \alpha, Y_{1:N}) = \frac{Q_{N,\alpha} f^2_k - (Q_{N,\alpha} f_k)^2}{\alpha + N + 1} \le \frac{1}{\alpha + N + 1}$, so that 
  \begin{math}
    \E\{\rho(P, Q_{N, \alpha}) \mid \alpha, Y_{1:N}\}
    \le
    \frac{1}{\sqrt{N + \alpha + 1}}.
  \end{math}
  By Markov's inequality and conditional expectation, we have
  \begin{align*}
    \Pr\{\rho(P, Q_{N,\alpha}) > \epsilon \mid Y_{1:N}\}
    \le
    \E\left\{\frac{1}{\epsilon \sqrt{N + \alpha + 1}} \mid Y_{1:N} \right\} \to 0
  \end{align*}
  as $N \to \infty$ by dominated convergence. It therefore suffices to show that $\Pr\{\rho(Q_{N,\alpha}, Q_t) > \epsilon \mid Y_{1:N}\} \to 0$ for all $\epsilon$. 

  Computing $\rho(Q_{N,\alpha}, Q_t)$ from the definition of $\rho(\cdot,\cdot)$ gives
  \begin{math}
    \rho(Q_{N,\alpha}, Q_t) \le w(t) \rho(\bbP_N, P_0) + 2 |w(\alpha / N) - w(t)|
    \text{ where } w(t) = (1 + t)^{-1}.
  \end{math}
  Consequently, we have
  \begin{align*}
    \Pr\{\rho(Q_{N,\alpha}, Q_t) > \epsilon \mid Y_{1:N}\}
    \le 1\{\rho(\bbP_N, P_0) > \epsilon/2\} + \Pr\{|w(\alpha/N) - w(t)| > \epsilon / 4 \mid Y_{1:N}\}.
  \end{align*}
  By Glivenko-Cantelli $\rho(\bbP_N, P_0) \to 0$ almost surely, while Theorem~\ref{thm:alpha} controls the second term because the posterior of $\alpha / N$ converges in distribution to $t$ with $P_0$-probability 1.
\end{proof}

\bibliographystyle{apalike}
\bibliography{mybib}

\end{document}